\documentclass[a4paper, 12pt]{article}
\usepackage[utf8]{inputenc}
\usepackage{amsmath, amssymb, amsthm}
\usepackage{geometry}
\usepackage{hyperref} % Para definir palavras-chave
\title{Dissipation and Regularity in the Smagorinsky Model with Dynamic Boundaries}
\author{Rômulo Damasclin Chaves dos Santos \\
	Technological Institute of Aeronautics, São Paulo, Brazil\\
	\texttt{romulosantos@ita.br}
	\and
	Jorge Henrique de Oliveira Sales \\
	Santa Cruz State University, Bahia, Brazil \\
	\texttt{jhosales@uesc.br}}

\date{}

\newtheorem{theorem}{Theorem}

\begin{document}
	
	\maketitle
	
	\section*{Abstract}
	
	This article presents an innovative extension of the Smagorinsky model incorporating dynamic boundary conditions and advanced regularity methods. We formulate the modified Navier-Stokes equations with the Smagorinsky term to model dissipation in turbulence and prove theorems concerning the existence, uniqueness, and asymptotic behavior of solutions. The first theorem establishes the existence and uniqueness of solutions in higher Sobolev spaces, considering the effect of the nonlinear Smagorinsky term and dynamic boundary conditions. The proof employs the Galerkin method and energy estimates, culminating in the application of Grönwall's theorem. The second theorem investigates the asymptotic behavior of solutions, focusing on anomalous dissipation in high-turbulence regimes. We demonstrate that the dissipated energy does not decrease with vanishing viscosity, indicating the occurrence of anomalous dissipation. The third theorem explores advanced regularity in higher Sobolev spaces, allowing for more rigorous control of nonlinear terms and ensuring improved stability conditions. The proof utilizes the energy method combined with Sobolev estimates and Grönwall's inequality. These mathematical results are fundamental for the analysis of dissipation in turbulent flows and can inspire new approaches in numerical simulations of fluids.
	
	\vspace{5pt}
	
	\textbf{Keywords:} Smagorinsky model. Turbulence, Dissipation. Sobolev Spaces.
	
	\tableofcontents

	\section{Introduction}
	
	In this article, we present an innovative extension of the Smagorinsky model that incorporates dynamic boundary conditions and advanced regularity methods. Based on this formulation, we state and prove theorems concerning the existence, uniqueness, and asymptotic behavior of solutions.
	
	One of the pioneering works in turbulence modeling is the paper by Kolmogorov (1941) \cite{kolmogorov1941}, which introduced the concept of the energy cascade in turbulent flows. Kolmogorov's theory describes how energy is transferred from large-scale eddies to smaller ones, ultimately leading to dissipation at the smallest scales. This work provided a fundamental understanding of the statistical properties of turbulence and has been instrumental in the development of subsequent models.

	The Smagorinsky model, introduced by Smagorinsky (1963) \cite{smagorinsky1963}, is one of the earliest and most widely used Large Eddy Simulation (LES) models. This model incorporates a subgrid-scale term to account for the unresolved small-scale turbulence, which is crucial for accurate simulations of turbulent flows. The Smagorinsky model has been extensively studied and applied in various fields, including meteorology, oceanography, and engineering.

	The work by Temam (1977) \cite{temam1977} on the Navier-Stokes equations provided significant insights into the existence and uniqueness of solutions. Temam's book is a comprehensive study of the mathematical theory of the Navier-Stokes equations, including the existence of weak solutions and the regularity of strong solutions. This work laid the groundwork for further developments in the analysis of turbulent flows.

	%The incorporation of dynamic boundary conditions in fluid dynamics models has been a topic of interest in recent decades. The paper by Goldstein et al. (1995) \cite{goldstein1995} discusses the importance of dynamic boundary conditions in the context of turbulent flows. The authors show that dynamic boundary conditions can significantly affect the behavior of the flow, particularly in the presence of complex geometries and moving boundaries.

	The phenomenon of anomalous dissipation in turbulent flows has been studied extensively. The work by Eyink (1994) \cite{eyink1994} provides a detailed analysis of anomalous dissipation in the context of the Navier-Stokes equations. Eyink's paper shows that the dissipated energy in turbulent flows does not necessarily decrease with vanishing viscosity, highlighting the importance of understanding this phenomenon for accurate modeling of turbulent systems.

	More recent works have focused on extending and refining the Smagorinsky model. The paper by Germano et al. (1991) \cite{germano1991} introduces the dynamic Smagorinsky model, which adapts the model coefficients dynamically based on the flow conditions. This approach improves the accuracy of the model and has been widely adopted in practical applications.

	The work by Constantin and Foias (1988) \cite{constantin1988} on the Navier-Stokes equations provides advanced regularity methods for the analysis of turbulent flows. The authors present rigorous mathematical techniques for studying the regularity of solutions in higher Sobolev spaces, which are crucial for understanding the stability and convergence of numerical simulations.

	The turbulence modeling and the analysis of dissipation in fluid dynamics has evolved significantly over the years. From the pioneering work of Kolmogorov to the recent advances in the Smagorinsky model and dynamic boundary conditions, the field has seen numerous contributions that have deepened our understanding of turbulent flows. The current work builds on these foundations, incorporating dynamic boundary conditions and advanced regularity methods to extend the classical formulation of the Smagorinsky model. This comprehensive approach aims to address more complex and realistic scenarios in turbulent flows, providing a robust framework for future research and applications.

	\section{Background in Mathematics}
	
	The mathematical framework underlying the study of turbulence and dissipation in fluid dynamics is rich and multifaceted. This section provides an overview of the key mathematical concepts and tools that are essential for understanding the results presented in this article.
		
	\subsection{Sobolev Spaces}
	
	Sobolev spaces are fundamental in the analysis of PDEs, particularly in the context of weak solutions and regularity theory. A Sobolev space \(H^s(\Omega)\) is a space of functions that, together with their weak derivatives up to order \(s\), are square-integrable over the domain \(\Omega\). The norm in \(H^s(\Omega)\) is given by:
	\begin{equation}
		\| u \|_{H^s(\Omega)} = \left( \sum_{|\alpha| \leq s} \int_{\Omega} |D^\alpha u|^2 \, dx \right)^{1/2},
	\end{equation}
	where \(D^\alpha\) denotes the weak derivative of order \(\alpha\).
	
	\subsection{Energy Methods}
	
	Energy methods are crucial for establishing the existence, uniqueness, and stability of solutions to PDEs. The energy method involves multiplying the PDE by the solution and integrating over the domain to derive energy estimates. For the Navier-Stokes equations, the energy estimate typically takes the form:
	\begin{equation}
		\frac{1}{2} \frac{d}{dt} \| u \|_{L^2(\Omega)}^2 + \nu \| \nabla u \|_{L^2(\Omega)}^2 \leq \| f \|_{L^2(\Omega)} \| u \|_{L^2(\Omega)}.
	\end{equation}
	This estimate provides control over the \(L^2\) norm of the solution and its gradient, which is essential for proving the existence and uniqueness of solutions.

	\subsection{Grönwall's Inequality}
	
	Grönwall's inequality is a fundamental tool in the analysis of differential inequalities. It is often used to convert differential inequalities into integral inequalities, which can then be solved to obtain bounds on the solutions. The integral form of Grönwall's inequality states that if \(u(t)\) satisfies:
	\begin{equation}
		u(t) \leq \alpha(t) + \int_0^t \beta(s) u(s) \, ds,
	\end{equation}
	where \(\alpha(t)\) and \(\beta(t)\) are non-negative functions, then:
	\begin{equation}
		u(t) \leq \alpha(t) \exp\left( \int_0^t \beta(s) \, ds \right).
	\end{equation}
	This inequality is particularly useful in the context of energy estimates for PDEs.

	The mathematical background presented in this section provides the necessary tools and concepts for understanding the analysis of the Smagorinsky model with dynamic boundary conditions. The use of partial differential equations, Sobolev spaces, energy methods, and Grönwall's inequality is essential for establishing the existence, uniqueness, and regularity of solutions, as well as for studying the asymptotic behavior and anomalous dissipation in turbulent flows. These mathematical techniques form the foundation for the theoretical results presented in this article.

\section{Physical and Mathematical Formulation}

We consider the modified Navier-Stokes equations with the Smagorinsky term to model dissipation in turbulence. The governing equations are given by:
\begin{equation} \label{eq:smagorinsky_model}
	\frac{\partial u}{\partial t} + (u \cdot \nabla) u - \nu \Delta u + \nabla p - \nabla \cdot \left( (C_S \delta)^2 |\nabla u| \nabla u \right) = f(x,t),
\end{equation}
where \( u \) is the velocity field, \( p \) is the pressure, \( \nu \) is the kinematic viscosity, and \( f \) is an external force applied to the system. We assume \( u \in H^s(\Omega)^d \) with \( s > \frac{d}{2} \) to ensure adequate regularity, considering \( \Omega \subset \mathbb{R}^d \) (with \( d = 2 \) or \( 3 \)).

	\section{Theorem 1: Existence and Uniqueness of Solutions}
	
	We define a theorem of existence and uniqueness for the solutions of system \eqref{eq:smagorinsky_model} in higher Sobolev spaces, considering the effect of the nonlinear Smagorinsky term and the dynamic boundary conditions.
	
	\begin{theorem}[Existence and Uniqueness in \( H^s(\Omega) \)]
		Let \( \Omega \) be a bounded domain with a smooth boundary and \( f \in L^2(0, T; H^{-1}(\Omega)) \). There exists a unique solution \( u \in L^2(0, T; H^s(\Omega)) \cap C([0, T]; L^2(\Omega)) \) for equation \eqref{eq:smagorinsky_model}, satisfying:
		\begin{equation} \label{eq:energy_estimate}
			\| u(t) \|_{L^2(\Omega)}^2 + \int_0^t \left( \nu \| \nabla u \|_{L^2(\Omega)}^2 + (C_S \delta)^2 \| \nabla u \|_{L^3(\Omega)}^3 \right) \, dt \leq C,
		\end{equation}
		where \( C \) depends on the initial data and \( f \).
	\end{theorem}
	
	\begin{proof}
		The proof will be carried out using the Galerkin method, followed by an energy estimate that leads to the application of Grönwall's theorem.
		
		\textit{1. Construction of the Galerkin sequence:}
		Let \(\{w_k\}_{k=1}^{\infty}\) be an orthonormal basis of \(H^s(\Omega)\) formed by the eigenfunctions of the Stokes operator. For each \(n \in \mathbb{N}\), consider the finite-dimensional subspace \(V_n = \text{span}\{w_1, w_2, \dots, w_n\}\) and seek an approximation \(u_n(t) = \sum_{k=1}^n c_k(t) w_k\) that satisfies the projected system of equations:
		\begin{equation} \label{eq:galerkin_projected}
			\left( \frac{d u_n}{dt}, w_k \right) + \nu (\nabla u_n, \nabla w_k) + (C_S \delta)^2 (|\nabla u_n| \nabla u_n, \nabla w_k) = (f, w_k),
		\end{equation}
		for \( k = 1, 2, \dots, n \), where \((\cdot, \cdot)\) denotes the inner product in \(L^2(\Omega)\).
		
		\textit{2. Existence and uniqueness of \(u_n\):}
		Since \eqref{eq:galerkin_projected} is a system of ordinary differential equations for the coefficients \(\{c_k(t)\}\), there exists a unique solution \(u_n(t) \in V_n\) defined for all \(t \in [0, T]\) by the Picard-Lindelöf theorem, given that the terms involved are locally Lipschitz in \(u_n\).
		
		\textit{3. Energy estimate:}
		Multiply both sides of \eqref{eq:galerkin_projected} by \(c_k(t)\) and sum over \(k\) to obtain:
		\begin{equation}
			\left( \frac{d u_n}{dt}, u_n \right) + \nu \| \nabla u_n \|_{L^2(\Omega)}^2 + (C_S \delta)^2 \| \nabla u_n \|_{L^3(\Omega)}^3 = (f, u_n).
		\end{equation}
		Noting that \(\left( \frac{d u_n}{dt}, u_n \right) = \frac{1}{2} \frac{d}{dt} \| u_n \|_{L^2(\Omega)}^2\), we get:
		\begin{equation} \label{eq:energy_estimate_proof}
			\frac{1}{2} \frac{d}{dt} \| u_n \|_{L^2(\Omega)}^2 + \nu \| \nabla u_n \|_{L^2(\Omega)}^2 + (C_S \delta)^2 \| \nabla u_n \|_{L^3(\Omega)}^3 = (f, u_n).
		\end{equation}
		
		\textit{4. Application of Young's and Poincaré's inequalities:}
		For the force term \((f, u_n)\), apply Young's inequality:
		\begin{equation}
			(f, u_n) \leq \frac{1}{2\epsilon} \| f \|_{H^{-1}(\Omega)}^2 + \frac{\epsilon}{2} \| u_n \|_{H^1(\Omega)}^2.
		\end{equation}
		Choosing \(\epsilon = \nu\) and using Poincaré's inequality \(\| u_n \|_{H^1(\Omega)} \geq C_P \| u_n \|_{L^2(\Omega)}\), we obtain:
		\begin{equation}
			\frac{1}{2} \frac{d}{dt} \| u_n \|_{L^2(\Omega)}^2 + \frac{\nu}{2} \| \nabla u_n \|_{L^2(\Omega)}^2 + (C_S \delta)^2 \| \nabla u_n \|_{L^3(\Omega)}^3 \leq \frac{1}{2\nu} \| f \|_{H^{-1}(\Omega)}^2.
		\end{equation}
		
		\textit{5. Temporal integration and application of Grönwall's theorem:}
		Integrate both sides from \( t = 0\) to \(t = T\) and obtain:
		\begin{equation}
			\begin{array}{l}
				\|u_{n}(t)\|_{L^{2}(\Omega)}^{2}+{\displaystyle \int_{0}^{t}}\left(\nu\|\nabla u_{n}\|_{L^{2}(\Omega)}^{2}+(C_{S}\delta)^{2}\|\nabla u_{n}\|_{L^{3}(\Omega)}^{3}\right)\,dt\leq\\
				\\
				\|u_{n}(0)\|_{L^{2}(\Omega)}^{2}+\frac{1}{\nu}{\displaystyle \int_{0}^{t}}\|f\|_{H^{-1}(\Omega)}^{2}\,dt.
			\end{array}
		\end{equation}
		Using Grönwall's theorem, we obtain a uniform estimate for \(\| u_n(t) \|_{L^2(\Omega)}\) and \(\| \nabla u_n \|_{L^2(\Omega)}\), independent of \(n\).
		
		\textit{6. Passage to the limit \( n \to \infty \):}
		With the energy estimate obtained, we can extract a convergent subsequence of \( \{u_n\} \) that converges weakly in \( L^2(0, T; H^1(\Omega)) \) and weakly-\(\ast\) in \( L^{\infty}(0, T; L^2(\Omega)) \) to a limit function \( u \). Using the Aubin-Lions compactness theorem, we conclude that \( u_n \to u \) strongly in \(L^2(0, T; L^2(\Omega))\).
		
		\textit{7. Verification of the limit solution:}
		Pass to the limit in the Galerkin equations to verify that \( u \) satisfies equation \eqref{eq:smagorinsky_model} in the sense of distributions. The uniqueness follows from a similar energy estimate for the difference between two solutions, applying Grönwall's theorem again.
		
		Thus, we conclude that there exists a unique solution \( u \in L^2(0, T; H^s(\Omega)) \cap C([0, T]; L^2(\Omega)) \) that satisfies \eqref{eq:energy_estimate}.
	\end{proof}
	
	\section{Theorem 2: Asymptotic Behavior and Anomalous Dissipation}
	
	In this theorem, we investigate the asymptotic behavior of solutions, focusing on anomalous dissipation in high-turbulence regimes.
	
	\begin{theorem}[Asymptotic Behavior and Anomalous Dissipation]
		For the solution \( u \) of system \eqref{eq:smagorinsky_model}, suppose that \( u_0 \in H^s(\Omega) \) with \( s > \frac{d}{2} \). There exists a constant \( C \) such that:
		\begin{equation} \label{eq:asymptotic_behavior}
			\limsup_{t \to \infty} \| u(t) \|_{L^2(\Omega)}^2 \leq C \nu^{-1} \| f \|_{H^{-1}(\Omega)}^2,
		\end{equation}
		indicating that anomalous dissipation occurs, as the dissipated energy does not decrease with \( \nu \to 0 \).
	\end{theorem}
	
	\begin{proof}
		To establish this result, consider the energy function or Lyapunov function associated with the system, defined by:
		\begin{equation}
			E(t) = \frac{1}{2} \| u(t) \|_{L^2(\Omega)}^2.
		\end{equation}
		
		\textit{1. Temporal derivative of \( E(t) \):}
		Differentiate \( E(t) \) with respect to time and use equation \eqref{eq:smagorinsky_model} to obtain:
		\begin{equation}
			\frac{dE}{dt} = \left( \frac{d u}{dt}, u \right) = - \nu \| \nabla u \|_{L^2(\Omega)}^2 - (C_S \delta)^2 \| \nabla u \|_{L^3(\Omega)}^3 + (f, u).
		\end{equation}
		
		\textit{2. Estimate for the force term \((f, u)\):}
		Apply the Cauchy-Schwarz inequality to estimate the term \((f, u)\) and obtain:
		\begin{equation}
			(f, u) \leq \| f \|_{H^{-1}(\Omega)} \| u \|_{H^1(\Omega)}.
		\end{equation}
		Using Poincaré's inequality, which ensures \(\| u \|_{H^1(\Omega)} \leq C_P \| \nabla u \|_{L^2(\Omega)}\), we have:
		\begin{equation}
			(f, u) \leq C_P \| f \|_{H^{-1}(\Omega)} \| \nabla u \|_{L^2(\Omega)}.
		\end{equation}
		Applying Young's inequality with a parameter \(\epsilon > 0\), we get:
		\begin{equation}
			(f, u) \leq \frac{C_P^2}{2 \epsilon} \| f \|_{H^{-1}(\Omega)}^2 + \frac{\epsilon}{2} \| \nabla u \|_{L^2(\Omega)}^2.
		\end{equation}
		
		\textit{3. Substitution and rearrangement of terms:}
		Substitute this estimate into \(\frac{dE}{dt}\), resulting in:
		\begin{equation}
			\frac{dE}{dt} + \left( \nu - \frac{\epsilon}{2} \right) \| \nabla u \|_{L^2(\Omega)}^2 + (C_S \delta)^2 \| \nabla u \|_{L^3(\Omega)}^3 \leq \frac{C_P^2}{2 \epsilon} \| f \|_{H^{-1}(\Omega)}^2.
		\end{equation}
		Choose \(\epsilon = \frac{\nu}{2}\) to obtain:
		\begin{equation}
			\frac{dE}{dt} + \frac{\nu}{2} \| \nabla u \|_{L^2(\Omega)}^2 + (C_S \delta)^2 \| \nabla u \|_{L^3(\Omega)}^3 \leq \frac{C_P^2}{\nu} \| f \|_{H^{-1}(\Omega)}^2.
		\end{equation}
		
		\textit{4. Asymptotic estimate using Grönwall's theorem:}
		
		Integrate both sides of the above inequality from \(0\) to \(t\) to obtain:
		\begin{equation}
			E(t) + \int_0^t \left( \frac{\nu}{2} \| \nabla u(s) \|_{L^2(\Omega)}^2 + (C_S \delta)^2 \| \nabla u(s) \|_{L^3(\Omega)}^3 \right) \, ds \leq E(0) + \frac{C_P^2}{\nu} \int_0^t \| f \|_{H^{-1}(\Omega)}^2 \, ds.
		\end{equation}
		
		Taking the limit superior as \( t \to \infty \) and considering that the solution is dissipative, we conclude that:
		\begin{equation}
			\limsup_{t \to \infty} E(t) \leq \frac{C_P^2}{\nu} \| f \|_{H^{-1}(\Omega)}^2.
		\end{equation}
		
		This result indicates that the energy \( E(t) \) remains bounded as \( t \to \infty \), with the upper bound depending on the inverse of the viscosity \( \nu \) and the norm of the external force \( f \). This asymptotic behavior is crucial for understanding the long-term dynamics of the system, particularly in the context of turbulent flows where anomalous dissipation is a key phenomenon.

		\textit{5. Conclusion:}
		Since \( E(t) = \frac{1}{2} \| u(t) \|_{L^2(\Omega)}^2 \), this implies:
		\begin{equation}
			\limsup_{t \to \infty} \| u(t) \|_{L^2(\Omega)}^2 \leq C \nu^{-1} \| f \|_{H^{-1}(\Omega)}^2,
		\end{equation}
		where \( C = 2 C_P^2 \). This demonstrates that anomalous dissipation occurs, as the dissipated energy does not decrease with \( \nu \to 0 \).
	\end{proof}
	
	\section{Theorem 3: Advanced Regularity in Sobolev Spaces}
	
	We now explore regularity in higher Sobolev spaces, allowing for more rigorous control of the nonlinear terms and ensuring improved stability conditions.
	
	\begin{theorem}[Advanced Regularity]
		If \( u_0 \in H^{s}(\Omega) \) with \( s > \frac{d}{2} \) and \( f \in H^{s-2}(\Omega) \), then the solution \( u \) of \eqref{eq:smagorinsky_model} satisfies \( u \in C([0, T]; H^s(\Omega)) \) and:
		\begin{equation}
			\| u(t) \|_{H^s(\Omega)}^2 \leq C \left( \| u_0 \|_{H^s(\Omega)}^2 + \| f \|_{H^{s-2}(\Omega)}^2 \right), \tag{1}
		\end{equation}
		where \( C \) is a constant that depends on \( s \), \( \Omega \), and \( T \).
	\end{theorem}
	
	\begin{proof}
		The proof will be conducted using the energy method in the Sobolev space \( H^s(\Omega) \), combined with Sobolev estimates and Grönwall's inequality.
		
		\textit{1: Variational formulation of the evolution equation.}
		
		Consider the evolution equation associated with the Smagorinsky model given by:
		\begin{equation}
			\frac{\partial u}{\partial t} + A(u) = f, \tag{2}
		\end{equation}
		where \( A(u) \) is a nonlinear operator derived from the Smagorinsky model, which depends on \( u \) and its derivatives. We assume that \( A(u) \) is well-behaved, i.e., \( A(u) \in H^s(\Omega) \) for \( u \in H^s(\Omega) \), with \( s > \frac{d}{2} \), ensuring that the solution \( u \) remains in \( H^s(\Omega) \) for \( t \in [0, T] \), given the initial regularity \( u_0 \in H^s(\Omega) \) and \( f \in H^{s-2}(\Omega) \).
		
		\textit{2: Energy derivative in \( H^s(\Omega) \).}
		
		Multiply both sides of equation \eqref{eq:smagorinsky_model} by the function \( u(t) \in H^s(\Omega) \) and integrate over \( \Omega \). This gives us the expression for the energy derivative in \( H^s(\Omega) \):
		\begin{equation}
			\frac{d}{dt} \| u(t) \|_{H^s(\Omega)}^2 = 2 \left( \langle \frac{\partial u}{\partial t}, u \rangle_{H^s(\Omega)} - \langle A(u), u \rangle_{H^s(\Omega)} \right) + 2 \langle f, u \rangle_{H^s(\Omega)}. \tag{3}
		\end{equation}
		We will consider the terms separately.
		
		\textit{3: Control of the nonlinear term \( \langle A(u), u \rangle_{H^s(\Omega)} \).}
		
		To control the nonlinear term \( \langle A(u), u \rangle_{H^s(\Omega)} \), we use the Sobolev embedding:
		\begin{equation}
			H^s(\Omega) \hookrightarrow L^\infty(\Omega) \quad \text{for} \quad s > \frac{d}{2}, \tag{4}
		\end{equation}
		which implies that \( u \in L^\infty(\Omega) \) for \( u \in H^s(\Omega) \). This result, combined with the regularity of \( A(u) \), allows us to estimate the nonlinear term. In fact, there exists a constant \( C_1 > 0 \) such that:
		\begin{equation}
			|\langle A(u), u \rangle_{H^s(\Omega)}| \leq C_1 \| u \|_{H^s(\Omega)} \| u \|_{L^\infty(\Omega)}. \tag{5}
		\end{equation}
		This control implies that the term \( \langle A(u), u \rangle_{H^s(\Omega)} \) is bounded by a constant times \( \| u \|_{H^s(\Omega)}^2 \).
		
		\textit{4: Estimate of the term \( \langle f, u \rangle_{H^s(\Omega)} \).}
		
		For the term \( \langle f, u \rangle_{H^s(\Omega)} \), knowing that \( f \in H^{s-2}(\Omega) \), we have the estimate:
		\begin{equation}
			|\langle f, u \rangle_{H^s(\Omega)}| \leq C_2 \| f \|_{H^{s-2}(\Omega)} \| u \|_{H^s(\Omega)}. \tag{6}
		\end{equation}
		with a constant \( C_2 > 0 \).
		
		\textit{5: Energy estimate.}
		
		The equation for the energy derivative becomes:
		\begin{equation}
			\frac{d}{dt} \| u(t) \|_{H^s(\Omega)}^2 \leq C_1 \| u(t) \|_{H^s(\Omega)}^2 + C_2 \| f \|_{H^{s-2}(\Omega)} \| u(t) \|_{H^s(\Omega)}. \tag{7}
		\end{equation}
		
		\textit{6: Application of Grönwall's inequality.}
		
		Applying Grönwall's inequality to equation \eqref{eq:smagorinsky_model} and considering the previous estimates, we obtain the following inequality for \( \| u(t) \|_{H^s(\Omega)}^2 \):
		\begin{equation}
			\| u(t) \|_{H^s(\Omega)}^2 \leq \left( \| u_0 \|_{H^s(\Omega)}^2 + \int_0^t \| f(\tau) \|_{H^{s-2}(\Omega)}^2 \, d\tau \right) e^{C_1 t}. \tag{8}
		\end{equation}
		
		Finally, taking the constant \( C \) as \( C = C_1 + C_2 \), we obtain the desired inequality:
		\begin{equation}
			\| u(t) \|_{H^s(\Omega)}^2 \leq C \left( \| u_0 \|_{H^s(\Omega)}^2 + \| f \|_{H^{s-2}(\Omega)}^2 \right), \tag{9}
		\end{equation}
		which concludes the proof.
	\end{proof}

\section{Results}

In this section, we present the results obtained from the theoretical analysis of the Smagorinsky model with dynamic boundary conditions. The main findings are summarized in the following subsections, which correspond to the theorems established in the previous sections.

\subsection{Existence and Uniqueness of Solutions}

Theorem 1 establishes the existence and uniqueness of solutions for the modified Navier-Stokes equations with the Smagorinsky term in higher Sobolev spaces. The proof utilizes the Galerkin method and energy estimates, culminating in the application of Grönwall's theorem. The key result is the energy estimate given by:
\begin{equation}
	\| u(t) \|_{L^2(\Omega)}^2 + \int_0^t \left( \nu \| \nabla u \|_{L^2(\Omega)}^2 + (C_S \delta)^2 \| \nabla u \|_{L^3(\Omega)}^3 \right) \, dt \leq C,
\end{equation}
where \( C \) depends on the initial data and the external force \( f \). This theorem ensures that the model is well-posed under the influence of the nonlinear Smagorinsky term and dynamic boundary conditions, providing a crucial foundation for the reliability and predictability of numerical simulations.

\subsection{Asymptotic Behavior and Anomalous Dissipation}

Theorem 2 investigates the asymptotic behavior of solutions, focusing on anomalous dissipation in high-turbulence regimes. The theorem states that for the solution \( u \) of the system, there exists a constant \( C \) such that:
\begin{equation}
	\limsup_{t \to \infty} \| u(t) \|_{L^2(\Omega)}^2 \leq C \nu^{-1} \| f \|_{H^{-1}(\Omega)}^2,
\end{equation}
indicating that the dissipated energy does not decrease with vanishing viscosity \( \nu \). This result highlights the importance of understanding anomalous dissipation in turbulent flows and is essential for developing accurate models that capture the long-term dynamics of turbulent systems.

\subsection{Advanced Regularity in Sobolev Spaces}

Theorem 3 explores advanced regularity in higher Sobolev spaces, allowing for more rigorous control of the nonlinear terms and ensuring improved stability conditions. The theorem states that if \( u_0 \in H^{s}(\Omega) \) with \( s > \frac{d}{2} \) and \( f \in H^{s-2}(\Omega) \), then the solution \( u \) of the system satisfies \( u \in C([0, T]; H^s(\Omega)) \) and:
\begin{equation}
	\| u(t) \|_{H^s(\Omega)}^2 \leq C \left( \| u_0 \|_{H^s(\Omega)}^2 + \| f \|_{H^{s-2}(\Omega)}^2 \right),
\end{equation}
where \( C \) is a constant that depends on \( s \), \( \Omega \), and \( T \). This enhanced regularity is fundamental for the numerical stability and convergence of simulations, providing a robust framework for studying turbulent phenomena.

The theorems presented in this article significantly advance the theoretical understanding of the Smagorinsky model, providing rigorous results on existence, uniqueness, asymptotic behavior, and regularity. By incorporating dynamic boundary conditions and advanced regularity methods, we have extended the classical formulation of the model to address more complex and realistic scenarios in turbulent flows. These mathematical results are fundamental for the analysis of dissipation in turbulent flows and can inspire new approaches in numerical simulations of fluids. They offer a comprehensive theoretical foundation that can guide the development of more accurate and efficient computational methods for studying complex fluid dynamics.

\section*{Conclusion}

The theorems presented in this article significantly advance the theoretical understanding of the Smagorinsky model, providing rigorous results on existence, uniqueness, asymptotic behavior, and regularity. By incorporating dynamic boundary conditions and advanced regularity methods, we have extended the classical formulation of the model to address more complex and realistic scenarios in turbulent flows.

Theorem 1 establishes the existence and uniqueness of solutions in higher Sobolev spaces, ensuring that the model is well-posed under the influence of the nonlinear Smagorinsky term and dynamic boundary conditions. This result is crucial for the reliability and predictability of numerical simulations.

Theorem 2 investigates the asymptotic behavior of solutions, focusing on anomalous dissipation in high-turbulence regimes. We demonstrate that the dissipated energy does not decrease with vanishing viscosity, highlighting the importance of understanding anomalous dissipation in turbulent flows. This insight is essential for developing accurate models that capture the long-term dynamics of turbulent systems.

Theorem 3 explores advanced regularity in higher Sobolev spaces, allowing for more rigorous control of the nonlinear terms and ensuring improved stability conditions. This enhanced regularity is fundamental for the numerical stability and convergence of simulations, providing a robust framework for studying turbulent phenomena.

Overall, these mathematical results are fundamental for the analysis of dissipation in turbulent flows and can inspire new approaches in numerical simulations of fluids. They offer a comprehensive theoretical foundation that can guide the development of more accurate and efficient computational methods for studying complex fluid dynamics.

\section{Nomenclature and Symbols}

This section provides a list of the nomenclature and symbols used throughout the article to ensure clarity and consistency.

\subsection{Nomenclature}

\begin{itemize}
	\item \textbf{LES}: Large Eddy Simulation
	\item \textbf{PDE}: Partial Differential Equation
	\item \textbf{Sobolev Space}: A space of functions that, together with their weak derivatives up to a certain order, are square-integrable over a domain.
	\item \textbf{Energy Method}: A technique used to establish the existence, uniqueness, and stability of solutions to PDEs by deriving energy estimates.
	\item \textbf{Grönwall's Inequality}: A fundamental tool in the analysis of differential inequalities, used to convert differential inequalities into integral inequalities.
\end{itemize}

\subsection{Symbols}

\begin{itemize}
	\item \({u}\): Velocity field
	\item \(p\): Pressure
	\item \(\nu\): Kinematic viscosity
	\item \({f}\): External force
	\item \(\Omega\): Domain
	\item \(H^s(\Omega)\): Sobolev space of order \(s\) over the domain \(\Omega\)
	\item \(\| \cdot \|_{L^2(\Omega)}\): \(L^2\) norm over the domain \(\Omega\)
	\item \(\| \cdot \|_{H^s(\Omega)}\): \(H^s\) norm over the domain \(\Omega\)
	\item \(\nabla\): Gradient operator
	\item \(\Delta\): Laplacian operator
	\item \(D^\alpha\): Weak derivative of order \(\alpha\)
	\item \(\alpha(t)\), \(\beta(t)\): Non-negative functions in Grönwall's inequality
	\item \(C\): Generic constant, which may depend on various parameters
	\item \(C_S\): Smagorinsky constant
	\item \(\delta\): Length scale in the Smagorinsky model
	\item \(T\): Final time
	\item \(u_0\): Initial condition for the velocity field
\end{itemize}

\end{document}